\documentclass[a4paper,reqno,11pt]{amsart}
\usepackage{comment}
\usepackage[T1]{fontenc}
\usepackage[utf8]{inputenc}
\usepackage{amssymb, amsmath, amsthm, graphicx, enumerate, enumitem, color}
\usepackage[dvipsnames,svgnames,table]{xcolor}
\usepackage[foot]{amsaddr}
\usepackage[normalem]{ulem}

\makeatletter
\def\thm@space@setup{
\thm@preskip=4mm
\thm@postskip=0mm
}
\makeatother

\usepackage{hyperref}
\usepackage[capitalise, compress, noabbrev]{cleveref}
\definecolor{linkblue}{named}{MidnightBlue}
\hypersetup{colorlinks=true, linkcolor=linkblue,  anchorcolor=linkblue,
	citecolor=linkblue, filecolor=linkblue, menucolor=linkblue,
	urlcolor=linkblue} 
\usepackage{mathtools}
\usepackage{thmtools, thm-restate}
\usepackage[longnamesfirst,numbers,sort&compress]{natbib}

\theoremstyle{plain}
\newtheorem{thm}{Theorem}
\newtheorem*{thm*}{Theorem}
\newtheorem{theorem}[thm]{Theorem}

\newtheorem{lemma}[thm]{Lemma}
\newtheorem*{lemma*}{Lemma}
\newtheorem{cor}[thm]{Corollary}
\newtheorem*{cor*}{Corollary}

\newtheorem{prop}[thm]{Proposition}
\newtheorem*{claim}{Claim}
\newtheorem*{lem*}{Lemma}

\newtheorem*{conjecture*}{Conjecture}

\newenvironment{proofclaim}[1][]
    {\let\oldqed\qedsymbol\renewcommand{\qedsymbol}{\ensuremath{\lozenge}}\begin{proof}[Proof of the claim] }{\end{proof}\renewcommand{\qedsymbol}{\oldqed}}

\let\ge\geqslant
\let\leq\leqslant
\let\geq\geqslant

\let\subset\subseteq

\let\epsilon\varepsilon

\DeclareMathOperator\pw{pw}
\DeclareMathOperator\tw{tw}
\DeclareMathOperator\td{td}

\DeclareMathOperator\ppw{ppw}

\definecolor{brightmaroon}{rgb}{0.76, 0.13, 0.28}
\newcommand{\defin}[1]{\emph{\textcolor{brightmaroon}{#1}}}

\setenumerate{label=\textup{(\roman*)}, noitemsep, topsep=-\parskip, 
labelindent=.2em, leftmargin=*, widest=iii,}
\setitemize{noitemsep, topsep=-\parskip, labelindent=.2em, leftmargin=*, widest=iii,}

\title{Blow-up structure of graphs excluding a tree or an apex-tree as a minor}

\begin{document}

\author[Claus]{Quentin Claus}
\address[Q.~Claus]{D\'epartement de Mathématiques, Universit\'e libre de Bruxelles, Belgium}
\email{quentin.claus@ulb.be}

\author[Joret]{Gwena\"el Joret}
\address[G.~Joret]{D\'epartement d'Informatique, Universit\'e libre de Bruxelles, Belgium}
\email{gwenael.joret@ulb.be}

\author[Rambaud]{Clément Rambaud}
\address[C.~Rambaud]{Universit\'e Côte d'Azur, CNRS, Inria, I3S, Sophia Antipolis, France}
\email{clement.rambaud@inria.fr}

\thanks{Q.\ Claus and G.\ Joret are supported by the Belgian National Fund for Scientific Research (FNRS)}

\begin{abstract}
    We prove blow-up structure theorems for graphs excluding a tree or an apex-tree as a minor. 
    First, we show that for every $t$-vertex tree $T$ with $t\geq 3$ and radius $h$, and every graph $G$ excluding $T$ as a minor, there exists a graph $H$ with pathwidth at most $2h-1$ such that $G$ is contained in $H\boxtimes K_{t-2}$ as a subgraph. This improves on a recent theorem of Dujmović, Hickingbotham, Joret, Micek, Morin, and Wood (2024), who proved the same result but with a larger bound on the order of the complete graph in the product. 
    
    Second, we show that for every $t$-vertex tree $T$ with $t\geq 2$,  radius $h$ and maximum degree $d$, and every graph $G$ excluding the apex-tree $T^+$ as a minor, where $T^+$ is the tree obtained by adding a universal vertex to $T$, there exists a graph $H$ with treewidth at most $4h-1$ such that $G$ is contained in $H\boxtimes K_{2(t-1)d}$.     
    The bound on the treewidth of $H$ is best possible up to a factor $2$, and improves on a $2^{h+2}-4$ bound that follows from a recent result of  Dujmović, Hickingbotham, Hodor, Joret, La, Micek, Morin, Rambaud, and Wood (2024). 
\end{abstract}

\maketitle

\section{Introduction}

In their \textit{Graph Minors} series of papers, Robertson and Seymour showed many results about the structure of graphs excluding a fixed graph as a minor.  
The very first of these results states that every graph $G$ excluding a $t$-vertex tree $T$ as a minor has pathwidth bounded by some function of $t$~\cite{robertson1983graph}.  
Later, Bienstock, Robertson, Seymour and Thomas \cite{bienstock1991quickly} (see also Diestel \cite{diestel1995graph}) 
showed an optimal upper bound of $t-2$ on the pathwidth of $G$:

\begin{theorem}[Bienstock, Robertson, Seymour and Thomas \cite{bienstock1991quickly}]
  \label{bienstock}
    Let $T$ be a tree on $t\geq 2$ vertices,  and let $G$ be a graph excluding $T$ as a minor. Then $\pw(G)\leq t-2$.
\end{theorem}

While an upper bound of $t-2$ on the pathwidth is best possible in general, Dujmović, Hickingbotham, Joret, Micek, Morin, and Wood~\cite{dujmovic2024excluded}  
recently showed that if the tree $T$ has radius $h$ (with possibly $h \ll t$), then a graph excluding $T$ as a minor is in fact ``not far'' from a graph of pathwidth $O(h)$: 

\begin{theorem}[Dujmovi{\'c}, Hickingbotham, Joret, Micek, Morin, Wood \cite{dujmovic2024excluded}]
\label{badtree}
    For every tree $T$ with $t$ vertices, radius $h$, and maximum degree $d$, for every $T$-minor-free graph $G$, there exists a graph $H$ such that $\pw(H) \leq 2h-1$ and $G$ is contained in $H \boxtimes K_{(d+h-2)(t-1)}$.\footnote{Here and in the rest of the paper, by `$G$ is contained in $J$' we mean that $G$ is {\em isomorphic} to a subgraph of the graph $J$.}    
\end{theorem}

Informally, the theorem shows that $G$ is then contained in a graph of pathwidth $2h-1$, where every vertex has been ``blown up'' by a clique whose size depends only on $T$. 
The upper bound of $2h-1$ on the pathwidth of $H$ is best possible, as shown in~\cite{dujmovic2024excluded}. 
However, determining the best possible bound on the size of the clique was left open. 
Our first result is the following improvement on the latter: 

\begin{restatable}{theorem}{restatetree} 
\label{tree}
     For every tree $T$ with $t\geq 3$ vertices and radius $h$, and for every $T$-minor-free graph $G$, there exists a graph $H$ such that $\pw(H) \leq 2h-1$ and $G$ is contained in $H \boxtimes K_{t-2}$.     
\end{restatable}

The bound of $t-2$ is almost optimal:  The graph $K_{t-1}$ is $T$-minor-free and yet, if $K_{t-1}$ is contained in $H\boxtimes K_c$ for some graph $H$ with $\pw(H)\leq 2h-1$, then we must have $c \geq (t-1)/2h$. This is because $t-2 = \pw(K_{t-1}) \leq \pw(H\boxtimes K_c) \leq  c(\pw(H)+1) - 1 \leq 2ch - 1$. 

Next, we turn our attention to graphs excluding a fixed graph $T^+$ as a minor, where $T$ is a tree and $T^+$ denotes the graph obtained from $T$ by adding a new vertex adjacent to all other vertices. 
Since $T^+$ is planar, it follows from the Grid-Minor Theorem of Robertson and Seymour~\cite{robertson1986graph} that graphs excluding $T^+$ as a minor have bounded treewidth. 
Later, Leaf and Seymour~\cite{leaf2015tree} gave the following bound on the treewidth: 

\begin{theorem}[Leaf and Seymour~\cite{leaf2015tree}]
\label{leafseymour}
Let $T$ be a tree and let $G$ be a $T^+$-minor-free graph.  Then $\tw(G)\leq \frac{3(|V(T^+)|-1)}2$. 
\end{theorem}

The upper bound on the treewidth was recently improved to $|V(T^+)|-2$ by Liu and Yoo~\cite{liu2025treewidthgraphexcludingapexforest}, which is best possible: 

\begin{theorem}[Liu and Yoo~\cite{liu2025treewidthgraphexcludingapexforest}]
\label{leafseymour-improved}
Let $T$ be a tree, and let $G$ be a $T^+$-minor-free graph.  Then $\tw(G)\leq |V(T^+)|-2$. 
\end{theorem}

Independently, we found that a slight modification of the proof of \cref{leafseymour} gives a different proof of \cref{leafseymour-improved} that is shorter than the one given in~\cite{liu2025treewidthgraphexcludingapexforest}.  This proof is given in the appendix.

Observe that \cref{leafseymour-improved} is similar to \cref{bienstock} 
except that treewidth is being bounded instead of pathwidth. 
It is natural to wonder whether a `blow-up' analogue of \cref{leafseymour-improved} exists, which would be similar to \cref{badtree}. 
In a loose sense, this is known to be the case, as follows from the following variant of the Grid Minor Theorem: 

\begin{theorem}[Dujmovi{\'c}, Hickingbotham, Hodor, Joret, La, Micek, Morin, Rambaud, Wood \cite{dujmovic2024grid}]
 \label{gridrevisited}
    For every planar graph $X$, there exists a positive integer $c$ such that for every $X$-minor-free graph $G$, there exists a graph $H$ of treewidth at most $2^{\td(X)}-4$ such that $G$ is contained in $H \boxtimes K_{c}$.  
\end{theorem}

In the above theorem, $\td(X)$ denotes the treedepth of $X$. 
Taking $X=T^+$ for some tree $T$ of radius $h$ in \cref{gridrevisited}, and using the fact that $T$ has treedepth at most $h+1$, and thus $T^+$ has treedepth at most $h+2$,  we obtain the following corollary: 

\begin{cor}[Dujmovi{\'c} {\it et al.} \cite{dujmovic2024grid}, implicit]
\label{cor:gridrevisited}
    For every tree $T$ of radius $h$, there exists a positive integer $c$ such that for every $T^+$-minor-free graph $G$, there exists a graph $H$ of treewidth at most $2^{h+2}-4$ such that $G$ is contained in $H \boxtimes K_{c}$. 
\end{cor}

Our second contribution is to show that the graph $H$ in \cref{cor:gridrevisited} can be chosen so that $H$ has treewidth $4h-1$, which is best possible up to a factor $2$, and answers positively a special case of Question~1 in~\cite{dujmovic2024grid}. 
We also obtain an explicit bound on the clique size $c$ in the blow-up, namely $c \leq 2(t-1)d$. 

\begin{theorem}
    \label{apexforest} 
    For every tree $T$ with $t\geq 2$ vertices, radius $h$, and maximum degree $d$, for every $T^+$-minor-free graph $G$, there exists a graph $H$ of treewidth at most $4h-1$ such that $G$ is contained in $H \boxtimes K_{(2t-1)d}$.
\end{theorem}

A straightforward modification of Proposition~3 in \cite{dujmovic2024excluded} shows that $H$ must have treewidth at least $2h$ in the above theorem, thus the upper bound of $4h-1$ is within a factor $2$ of optimal. 
We suspect that the upper bound of $(2t-1)d$ on the clique size in the blow-up in \cref{apexforest} could be further improved to $O(t)$, as in \cref{tree}, however the proof of the latter result does not seem to be easily adaptable to this setting. 

The paper is organized as follows. 
In \cref{sec:background}, we give the necessary definitions and notation. 
In \cref{sec:excluding_tree}, we prove \cref{tree}.  
Next, in \cref{sec:excluding_apex_tree}, we prove \cref{leafseymour-improved} and \cref{apexforest}. 
Finally, in \cref{sec:open}, we conclude with some open problems.

\section{Preliminaries}\label{sec:background}
We consider simple, finite, undirected graphs.
Given a graph $G$, we denote by $V(G)$ its vertex set, and by $E(G)$ its edge set. 
For $c\in \mathbb N$, we denote by $K_c$ the complete graph on $c$ vertices. 
For $h, d \in \mathbb N$, we denote by $T_{h, d}$ the complete $d$-ary tree of radius $h$.  
For a tree $T$, we denote by $T^+$ the apex-tree obtained by adding a universal vertex to $T$. 

A \defin{rooted tree} is a tree where a vertex is specified to be the root.  The \defin{height} of a rooted tree is the maximum distance between the root and a vertex of the tree. 
A \defin{rooted forest} is a disjoint union of rooted trees.  
A \defin{vertical path} of $F$ is a path contained in some rooted tree $T$ of $F$ having the root of $T$ as one of its endpoints.   

For a graph $G$, and a set $S\subseteq V(G)$, we define the \defin{neighborhood} of $S$ in $G$ as the set of vertices in $V(G) \setminus S$ that are adjacent to at least one vertex in $S$, and we denote it by $N_G(S)$.  If $S$ contains only one vertex $v$, we sometimes write $N_G(v)$ instead of $N_G(\{v\})$. 
We will denote by $\partial_G(S)$ the set of vertices in $S$ that are adjacent to at least one vertex not in $S$. 
For both notations $N_G(S)$ and $\partial_G(S)$, we drop the subscript when the graph is clear from the context. 

For a graph $G$, and set $S\subseteq V(G)$, we denote by $G-S$ the subgraph of $G$ obtained by removing all the vertices of $G$ contained in $S$.  If $S$ contains only a vertex $v$, we simply write $G-v$ instead of $G-\{v\}$.  

A graph $H$ is a \defin{minor} of a graph $G$ if $H$ can be obtained from a subgraph of $G$ by contracting edges.  A graph $G$ is $\mathit{H}$\defin{-minor-free} if $H$ is not a minor of $G$. An $\mathit{H}$\defin{-model} in a graph $G$, also called \defin{a model of} $H$ in $G$, consists of pairwise-disjoint vertex sets $\{W_v \subseteq V(G) \mid v \in V(H)\}$ that each induce a connected subgraph in $G$, and such that for every edge $vw$ of $H$, then there exists $x\in W_v$ and $y\in W_w$ such that $x$ and $y$ are adjacent in $G$.  Note that $H$ is a minor of $G$ if and only if $G$ contains an $H$-model.  

A \defin{tree-decomposition} of a graph $G$ is a collection $\mathcal D:=\{B_x\subseteq V(G) \mid x\in V(T)\}$ of vertex subsets of $G$ called \defin{bags}, where $T$ is a tree, and such that 
\begin{itemize}
    \item for each vertex $v\in V(G)$, the set $\{x \in V(T) \mid v\in B_x\}$ induces a nonempty subtree of $T$, and 
    \item for each edge $vw\in E(G)$, there exists $x\in V(T)$ such that both $v$ and $w$ are contained in the bag $B_x$.
\end{itemize} 
We will say that $T$ is the tree \defin{associated} to this decomposition, and for every vertex $x\in V(T)$, that $B_x$ is the bag \defin{associated} to $x$. 
The \defin{width} of the tree-decomposition is $\max\{|B_x| \mid x\in V(T)\}-1$.  The \defin{treewidth} of a graph $G$, denoted by $\tw(G)$, is the minimal width of a tree-decomposition of $G$.  

A \defin{path-decomposition} of a graph $G$, the \defin{width} of a path-decomposition and the \defin{pathwidth} of a graph $G$, denoted by $\pw(G)$, are defined the same way except that the tree $T$ is required to be a path.  
It will be convenient to denote path-decompositions simply as a sequence $B_1, B_2, \dots, B_m$ of bags.  In this case, we will call $B_1$ the first bag of the decomposition. 

Let $G$ be a graph.  
A \defin{partition} of $V(G)$ is a collection $\mathcal P$ of nonempty subsets of $V(G)$ such that every vertex of $G$ is in exactly one set of the collection. 
Each element of $\mathcal P$ is called a \defin{part}.  The \defin{width} of $\mathcal{P}$ is the maximum size of a part of $\mathcal P$. The \defin{quotient} $G/\mathcal P$ is the graph with vertex set $\mathcal{P}$ and where distinct parts $P,P'\in \mathcal{P}$ are adjacent in $G/\mathcal P$ if and only if there exists a vertex of $P$ which is adjacent to a vertex of $P'$ in $G$. 

The \defin{strong product} of two graphs $G$ and $H$, denoted by $G\boxtimes H$, is the graph with vertex set $V(G\boxtimes H):= V(G) \times V(H)$ and such that two distinct vertices $(v, w)$ and $(v', w')$ in $V(G\boxtimes H)$ are adjacent if and only if the following two conditions are verified:
\begin{enumerate}
    \item $v=v'$, or $v$ and $v'$ are adjacent in $G$
    \item $w=w'$, or $w$ and $w'$ are adjacent in $H$.
\end{enumerate}

We will use the following easy observation throughout the rest of the paper (often implicitly):  

\begin{prop}
    Given two graphs $G, H$ and a positive integer $c$, the graph $G$ is contained in $H\boxtimes K_c$ if and only if there exists a partition $\mathcal P$ of $V(G)$ of width at most $c$ such that $G/\mathcal P$ is isomorphic to a subgraph of $H$.  
\end{prop}

\section{Excluding a tree} \label{sec:excluding_tree}
In this section we prove \cref{tree}, which we restate here for convenience. 

\restatetree* 

Let us recall that the statement of this result is the same as that of \cref{badtree} proved in~\cite{dujmovic2024excluded}, except the clique size in the blow-up is decreased from $(d+h-2)(t-1)$ to $t-2$. In \cite{dujmovic2024excluded}, the authors derive \cref{badtree} as a consequence of the following lemma (Lemma~8 in \cite{dujmovic2024excluded}). 

\begin{lemma}[\citet{dujmovic2024excluded}]
\label{lemmearticle}
    For every $h, d\in \mathbb N$ with $h+d\geq 3$, 
    for every $T_{h, d}$-minor-free graph $G$, for every vertex $r$ in $V(G)$, the graph $G$ has a partition $\mathcal P$ of width at most $(d+h-2)(\pw(G)+1)$ such that $\{r\} \in \mathcal P$, and $G/\mathcal{P}$ has a path-decomposition of width at most $2h-1$ such that the first bag contains $\{r\}$.    
\end{lemma}

To see that \cref{lemmearticle} implies \cref{badtree}, observe that if $G$ excludes a $t$-vertex tree $T$ of radius $h$ and maximum degree $d$ as a minor then in particular $G$ excludes $T_{h, d}$ as a minor, and furthermore $\pw(G) \leq t-2$ by \cref{bienstock}. 

The heart of the proof of \cref{lemmearticle} in \cite{dujmovic2024excluded} is the following technical lemma, which is not stated explicitly in \cite{dujmovic2024excluded} but follows from their proofs. 
For completeness, we note that this lemma can also be derived as a special case of \cref{metalemmeproduit} (with $R=V(G)$)  introduced later in this paper. 

\begin{lemma}[\citet{dujmovic2024excluded}, implicit]
    \label{metalemme}
    Let $G$ be a graph, let $x, y\in \mathbb N$, and let $r\in V(G)$.  Let $R:=N_G(r)$.  Suppose that there exists a set $X\subseteq V(G) \backslash \{r\}$ of size at most $x$ satisfying the following conditions: 
    \begin{enumerate}
        \item For each vertex $v\in X$, there exists a path in $G$ from $r$ to $v$ internally disjoint from $X$.
        \item For each connected component $C$ of $G-X-r$ containing a vertex of $R$, there exists a partition $\mathcal P_C$ of $V(C)$ of width at most $x$ such that $\pw(C/\mathcal{P}_C)\leq y$.
        \item If $G'$ is a minor of $G$ with $|V(G')|<|V(G)|$ and there exists a $G'$-model in $G$ such that $r$ is contained in the branch set of a 
        vertex $r'$ of $G'$, then there exist a partition $\mathcal{P}_{G'}$ of $V(G')$ of width at most $x$ such that $\{r'\}\in \mathcal{P}_{G'}$, and a path-decomposition $\mathcal{D}'$ of $G'/\mathcal P_{G'}$ of width at most $y+2$ such that $\{r'\}$ is in the first bag of $\mathcal{D}'$. 
    \end{enumerate}
    Then there exist a partition $\mathcal P$ of $V(G)$ of width at most $x$ such that $\{r\} \in \mathcal P$, and a path-decomposition $\mathcal D$ of $G/\mathcal P$ of width at most $y+2$ such that $\{r\}$ is in the first bag of $\mathcal D$. 
\end{lemma}

In order to prove \cref{tree}, we will reuse the ideas from \cite{dujmovic2024excluded}, and in particular \cref{metalemme}, in combination with the following lemma, which is a slight modification of a recent result in \cite{dujmovic2025tight}. 
To state this lemma, we need the following terminology: Given a rooted forest $F$, a graph $G$, and a set $R$ of vertices of $G$, we say that an $F$-model in $G$ is \defin{weakly} $\mathit{R}$-\defin{rooted} if the branch set of each of the roots of the trees of $F$ contains at least one vertex of $R$. To keep notations light, it will be convenient to also allow $R$ to contain vertices not in $G$ in this definition, and thus we do so. (This allows us for instance to write ``weakly $R$-rooted'' instead of ``weakly $(R\setminus X)$-rooted'' in the following lemma.) 

\begin{lemma}
\label{newbound}
    Let $G$ be a graph, let $R\subseteq V(G)$, and let $F$ be a rooted forest.  If there is no weakly $R$-rooted model of $F$ in $G$, then there exists a connected component $T$ of $F$ and a set $X\subseteq V(G)$ with $|X|< |V(F)|$ such that $G-X$ does not contain any weakly $R$-rooted model of $T$.    
\end{lemma}

Before proving \cref{newbound}, let us we show that it implies the following lemma, which directly implies \cref{tree}, by applying it with an arbitrary vertex $r$ on every connected component of $G$, and choosing the root of $T$ in such a way that the height of $T$ is the radius of $T$.

\begin{lemma} \label{lemmearticlevariante} 
    For every rooted tree $T$ with $t\geq 3$ vertices and height $h\geq 1$, for every connected graph $G$, for every $r \in  V(G)$ such that $G$ does not contain any weakly $\{r\}$-rooted model of $T$, there is a partition $\mathcal P$ of $V(G)$ of width at most $t-2$ such that $\{r\} \in \mathcal P$, and $G/\mathcal{P}$ has a path-decomposition of width at most $2h-1$ such that the first bag contains $\{r\}$.    
\end{lemma}
\begin{proof}[Proof of \cref{lemmearticlevariante}, assuming \cref{newbound}]
    We prove the lemma by induction on $(h, |V(G)|)$, in lexicographic order.  
    The base case $h=1$ is done similarly as in the proof of \cref{badtree} in \cite{dujmovic2024excluded}, we describe it here for completeness.  For every $i\in \mathbb N$, let $V_i := \{v\in V(G) \ | \ \text{dist}_G(r, v)=i \}$.  If $|V_i|\geq t-1$ for some $i\in \mathbb N$, then the set $V_0\cup V_1\cup \ldots \cup V_{i-1}$ (which is connected in $G$) together with $t-1$ vertices of $V_i$ taken as singletons give a weakly $\{r\}$-rooted $T$-model in $G$, a contradiction.    
    Thus, $|V_i|\leq t-2$, for all $i\in \mathbb N$.   Let $m$ be the largest integer such that $V_{m}\ne \emptyset$. Because $G$ is connected, $\{V_i \ | \ i\in \mathbb N, i\leq m \}$ is a partition of $V(G)$ of width at most $t-2$.  Moreover, $V_0=\{r\}$, and it is easy to verify that $\big\{\{V_{i}, V_{i+1} \} \ | \ 0\leq i < m\big\}$ is a path decomposition of $G/ \mathcal P$ of width at most $1=2h-1$, where $\{r\}=V_0$ is in the first bag.  This concludes the base case.   

    Now for the inductive case, assume $h\geq 2$. 
    First, we show that the hypotheses of \cref{metalemme} are fulfilled for $x=t-2$ and $y=2h-3$.  Let $F$ be the rooted forest obtained from $T$ by deleting its root, where for each connected component of $F$, the root of that tree is the only vertex that was a child of the root of $T$. 
    Let $R:=N_G(r)$. 
    Observe that $G-r$ cannot contain a weakly $R$-rooted model of $F$, because otherwise $G$ would contain a weakly $\{r\}$-rooted model of $T$.  
    Applying \cref{newbound} to $G-r$, we obtain that there exists a connected component $T'$ of $F$ and a set $X\subseteq V(G-r)$ of size at most $t-2$ such that $G-(X \cup \{r\})$ has no weakly $R$-rooted $T'$-model.  Without loss of generality, we may choose $X$ to be inclusion-wise minimal with this property.  Notice also that by definition of $F$, the height $h'$ of $T'$ is strictly less than $h$. 

    Let us point out that the set $X$ could be empty, which is fine for our purposes. 
    In case $X$ is not empty, by minimality of $X$, for each vertex $v\in X$ the graph $G-r-(X \setminus \{v\})$ contains a weakly $R$-rooted model of $T'$.  This model has to use $v$, so in particular it contains a path from $v$ to a vertex of $R$, and thus $G-(X \setminus \{v\})$ contains a path from $v$ to $r$.  This path has to be internally disjoint from $X$, so the first condition of \cref{metalemme} is verified. 

    Next, suppose that $C$ is a connected component of $G-(X \cup r)$ containing a vertex $v\in R$.  By definition of $X$, $C$ does not contain any weakly $\{v\}$-rooted model of $T'$.  Thus, since $h'<h$, we may apply induction, and by the induction hypothesis, the second condition of \cref{metalemme} is verified for the set $X$ as well. 
    
    Finally, suppose that $G'$ is a minor of $G$ such that $|V(G')|<|V(G)|$, and such that there exists a model of $G'$ in $G$ and a vertex $r'$ in $G'$  whose branch set in the model contains $r$.  
    If there is a weakly $\{r'\}$-rooted model of $T$ in $G'$, then there is a weakly $\{r\}$-rooted model of $T$ in $G$, a contradiction.  
    So there is no weakly $\{r'\}$-rooted model of $T$ in $G'$, and because $G'$ has less vertices than $G$, we may apply the induction hypothesis, and we deduce that the third condition of \cref{metalemme} is also verified. 

    Therefore, we may apply \cref{metalemme} to $G$ and $X$, which yields the desired result.
\end{proof}

Now we prove \cref{newbound}. We will need the following lemma due to \citet{diestel1995graph}. 
We remark that it is not stated explicitly in \cite{diestel1995graph} but it follows from the main proof in that paper. 

\begin{lemma}[\citet{diestel1995graph}, implicit]
    Let $G$ be a connected graph, let $v\in V(G)$, and let $T$ be a rooted tree on $t>1$ vertices.  If $\pw(G)\geq t-1$, then there exists $Y\subseteq V(G)$ such that $r\in Y$, $G[Y]$ contains a weakly $\{r\}$-rooted $T$-model, and $G[Y]$ has a path-decomposition of width at most $t-1$ whose last bag contains $\partial_G(Y)$.
    \label{Diestel}
\end{lemma}

The following lemma, whose proof depends on \cref{Diestel}, is a slight modification of a result in~\cite{dujmovic2025tight} (stated at the beginning of the proof of Theorem 2 in that paper).  The only difference is that, in~\cite{dujmovic2025tight}, the models are not rooted. We include a proof of the lemma for completeness; we emphasize that it is the same proof as in \cite{dujmovic2025tight}, except for minor changes to make sure that the models we build are rooted in the appropriate way. 

\begin{lemma}
Let $G$ be a graph, let $R\subseteq V(G)$, let $c$ be a positive integer, let $t_1, \ldots, t_c$ be positive integers with $t_1 \leq \ldots \leq t_c$,  let $T_1, \ldots ,T_c$ be rooted trees such that $|V(T_i)|=t_i$ for every $i\in [c]$, let $x_1, \ldots , x_c$ be nonnegative integers at least one of which is nonzero, and let $I:= \{i \in \mathbb N \mid 1\leq i \leq c \text{ and } x_i\geq 1 \}$.  Then, at least one of the two following conditions are satisfied 
    \begin{enumerate}
        \item $G$ contains pairwise vertex-disjoint subgraphs $\{M_{i, j} \ | \ 1\leq i\leq c, 1\leq j\leq x_i\}$ such that for each $i\in [c]$ and each $j\in [x_i]$, $M_{i, j}$ contains a weakly $R$-rooted $T_i$-model. 
        \item There exists $X\subseteq V(G)$ with $|X|\leq (\sum_{i \in I}x_it_i)- t_{\max(I)}$, and $i\in I$ such that $G- X$ does not contain any weakly $R$-rooted $T_i$-model. 
    \end{enumerate}
    \label{ErdosPosaForest}
\end{lemma}

\begin{proof}
    We call the tuple $(G, R, c, T_1, \ldots, T_c, x_1, \ldots ,x_c)$ an \defin{instance}.
    
    Let $(G, R, c, T_1, \ldots, T_c, x_1, \ldots, x_c)$ be an instance, and let $m:= \min(I)$. The proof is by induction on ($\sum_{i=1}^c x_i, |V(G)|)$, in lexicographic order.  In the base case, $\sum_{i=1}^c x_i=1$, implying that $x_m=1$ and $x_i=0$ for every $i\in \{1, \ldots, c\} \setminus\{m\}$.  Either there is a weakly $R$-rooted model of $T_m$ in $G$, and the first statement holds with $M_{m, 1}=G$, or there is no $R$-rooted model of $T_m$ is $G$, and then the second statement holds with $X=\emptyset$. 
    
    For the inductive case, assume that $\sum_{i=1}^c x_i\geq 2$, and that the statement holds for strictly smaller values of the sum, or equal values of the sum and strictly smaller values of the number of vertices. 
    If, for every integer $i\in I$, there is no weakly $R$-rooted model of $T_i$ in $G$, then the second statement holds with $X=\emptyset$.  Thus, we may assume that there exists $i\in I$ such that $G$ contains an $R$-rooted model of $T_i$.  If there exists a connected component $C$ of $G$ such that $V(C)\cap R=\emptyset$, then we apply induction on the instance $(G-V(C), R\backslash V(C), c, T_1, \ldots, T_c, x_1, \ldots, x_c)$.  If the first statement holds for this instance, then it holds also for $(G, R, c, T_1, \ldots, T_c, x_1, \ldots, x_c)$.  If the seconds statement holds for $(G-V(C), R\backslash V(C), c, T_1, \ldots, T_c, x_1, \ldots x_c)$, then it also holds for $(G, R, c, T_1, \ldots, T_c, x_1, \ldots x_c)$ with the same set $X$.  Thus we may assume that every connected component of $G$ contains at least a vertex of $R$.
    
    If $\pw(G) < t_m-1$, let $Y:= V(G)$.  Notice that in this case $\partial_G(Y)=\emptyset$.
    If $\pw(G)\geq t_m-1$, then there exists a connected component  $C$ of $G$ such that $\pw(C)\geq t_m-1$. 
    Then we apply \cref{Diestel} on $C, T_m$ and an arbitrary vertex $r\in R\cap V(C)$. 
    In this case, let $Y$ be the resulting subset of vertices.
    
    Notice that, in both cases in the definition of $Y$, the following two properties hold: 
    \begin{itemize}
        \item $G[Y]$ contains a weakly $R$-rooted model of $T_i$ for some $i\in I$, and 
        \item $G[Y]$ has a path-decomposition $(B_1, B_2, \ldots, B_q)$ of width at most $t_m-1$ such that $\partial_G(Y)\subseteq B_q$.  
    \end{itemize}
    (In case $Y=V(G)$, the first property holds because it holds for $G$.) 
    Let $\ell$ be the smallest integer $j$ between $1$ and $q$ such that there exists $i\in I$ such that $G_{\ell}:= G[\bigcup_{j=1}^\ell B_j]$ contains a weakly $R$-rooted 
    model of $T_i$.  Let $i'\in I$ be such that $G_\ell$ contains a weakly $R$-rooted model of $T_{i'}$. 
    
    We claim that there is no edge between any vertex of $G_\ell - B_\ell$ and any vertex of $G-V(G_{\ell})$.  Suppose for a contradiction that there exists an edge $uv\in E(G)$ with $u\in V(G_\ell)\setminus B_\ell$ and $v\in V(G) \setminus V(G_{\ell})$.  First, notice that $u\in \bigcup_{j=1}^{\ell-1}B_j$.  If $v\in Y$, then $v$ appears only in bags of $(B_1, \ldots, B_q)$ with indices strictly bigger than $\ell$.  However, $u$ appears in at least one bag of index strictly smaller than $\ell$.  Then, it follows from the definition of a path-decomposition that $u\in B_{\ell}$, a contradiction.  Thus, $v\not \in Y$, and then $u\in \partial_G(Y)$ and so $u\in B_q$.  Because $u$ appears in $B_q$ as well as in a bag of index strictly smaller than $\ell$, we should have $u\in B_{\ell}$, a contradiction. 
    
    Let $G':= G-V(G_\ell)$.  Let $x'_i:= x_i$ for every $i\in \{1, \ldots, c\} \backslash \{i'\}$, and let $x'_{i'}:= x_{i'}-1$.  Let $I':= \{i\in I \ | \ x'_i\geq 1\}$.  Observe that $I'$ is not empty, since $\sum_{i\in I} x_i \geq 2$. 
    Apply induction to the instance $(G', R\backslash V(G_l), c, T_1, \ldots, T_c, x'_1, \ldots, x'_c)$.  
    If there exists a collection of vertex-disjoint subgraphs $\{M'_{i, j} \ | \ i\in I', 1\leq j \leq x'_i\}$ such that $M'_{i, j}$ contains a weakly $R$-rooted  
    model of $T_i$ for every $i\in I'$ and $j\in \{1, \ldots, x'_i\}$, then setting $M_{i, j}:=M'_{i, j}$ for every $i, j$ such that $1\leq i \leq c$, $i\ne i'$ and $1\leq j\leq x_i$, or $i=i'$ and $1\leq j \leq x_i - 1$ and setting $M_{i', x_{i}}=G_{\ell}$ satisfies the first statement of the lemma. 
    
    Otherwise, there is a set $X'$ of size at most $(\sum_{i\in I'} x'_it_i)-t_{\max(I')}$ and an integer $z\in I'$ such that $G'-X'$ does not contain any weakly $R$-rooted 
    model of $T_z$.  Let $X:= X'\cup B_{\ell}$.  
    We have 
    \[
    |X|=|X'|+|B_{\ell}|\leq \left(\left(\sum_{i\in I'} x'_it_i\right)-t_{\max(I')}\right)+t_m=\left(\sum_{i\in I}x_it_i\right)-(t_{\max(I')}+t_{i'}-t_m).
    \]
    If $\max(I')=\max(I)$, then $t_{\max(I')}+t_{i'}-t_m=t_{\max(I)}+(t_{i'}-t_m)\geq t_{\max(I)}$ because $m=\min(I)$, and thus $i'\geq m$ and so $t_{i'}\geq t_m$. 
    
    If $\max(I')\ne \max(I)$, then $\max(I)=i'$ (and $x_{i'}=1$), and then $t_{\max(I')}+t_{i'}-t_m=t_{\max(I)}+(t_{\max(I')}-t_m)\geq t_{\max(I)}$, because $\max(I')\in I$, thus $\max(I')\geq m$, and thus $t_{\max(I')}\geq t_m$. 
    
    Thus, in both cases, we have $t_{\max(I')}+t_{i'}-t_m \geq t_{\max(I)}$, and hence $|X|\leq \left(\sum_{i\in I}x_i t_i\right) - t_{\max(I)}$. 
    To fulfill the second statement, it suffices to show that $G-X$ does not contain any weakly $R$-rooted model of $T_z$. 
    Because $B_{\ell}\subseteq X$, and since we have already shown that there is no edge in $G$ between any vertex of $G_\ell-B_\ell$ and any vertex of $G-V(G_{\ell})$, it suffices to show that there is no weakly $R$-rooted  model of $T_z$ 
    in $G_{\ell}-B_\ell$ and no weakly $R$-rooted 
    model of $T_z$ in $G-V(G_\ell)-X$. 
    The first property holds because of the definition (in particular, the minimality) of $\ell$.  
    The second property holds because there is no weakly $R$-rooted 
    model of $T_z$ in $G'-X'$.  
    This concludes the proof.
\end{proof}

\cref{newbound} follows from \cref{ErdosPosaForest} by letting $T_1, \dots, T_c$ be the connected components of the forest $F$, and letting $x_1=\cdots=x_c=1$. 
This concludes the proof of \cref{tree}.

\section{Excluding an apex-forest} \label{sec:excluding_apex_tree}
 
Let us define some notions introduced in \cite{hodor2024quickly} (see also \cite{claus2026excludingapexforestfanquickly}) that we will use in this section: Let $G$ and $H$ be graphs and let $S, R \subseteq V(G)$.  An $H$-model in $G$ is said to be $\mathit{S}$\defin{-rooted} if the branch set of each vertex of $H$ contains at least one vertex of $S$.  If $H$ is a rooted forest, an $H$-model in $G$ is said to be $\mathit{(S, R)}$\defin{-rooted} if it is $S$-rooted and weakly $R$-rooted. 
To keep notations light when considering subgraphs, it will be convenient to also allow $S$ and $R$ to contain vertices not in $G$ in this definition, and thus we do so.
A \defin{path-decomposition of} $\mathit{(G, S)}$ consists of an induced subgraph $H$ of $G$ such that $S\subseteq V(H)$ and a path-decomposition $\mathcal{D}$ of $H$ such that for every connected component $C$ of $G-V(H)$, there exists a bag of $\mathcal D$ containing all of $N_{G}(V(C))$.  
The \defin{pathwidth of} $\mathit{(G, S)}$, denoted by $\pw(G, S)$, is the minimum width of a path-decomposition of $(G, S)$.

The authors of \cite{hodor2024quickly} proved the following theorem.  
  
 \begin{theorem}[\citet{hodor2024quickly}]
    For every forest $F$ with at least one vertex, for every graph $G$ and for every $S\subseteq V(G)$, if $G$ has no $S$-rooted model of $F$, then $\pw(G, S)\leq 2|V(F)|-2$.
    \label{diestelSrootedoriginal}
\end{theorem}

Using a slight modification of their proof, one can obtain the following rooted variant of their theorem, which we will use in this section: 

\begin{theorem}
    For every rooted tree $T$ with at least one vertex, for every connected graph $G$ and for every $S, R\subseteq V(G)$ such that $R$ is nonempty,   
    if $G$ has no $(S, R)$-rooted model of $T$, then $\pw(G, S)\leq 2|V(T)|-2$.
    \label{diestelSrooted}
\end{theorem}


Next, taking the point of view of graph blow-ups,  we introduce a generalization of the definition of the pathwidth of $(G, S)$ which will be helpful for our purposes.  
Let $G$ be a graph, let $S\subseteq V(G)$, and let $k$ be a positive integer. A $\mathit{k}$-\defin{partition-path-decomposition of} $\mathit{(G, S)}$ consists of an induced subgraph $H$ of $G$ such that $S\subseteq V(H)$, a partition $\mathcal P_H$ of $V(H)$ of width at most $k$, and a path-decomposition $\mathcal D_H$ of $H/\mathcal{P_H}$ such that for every connected component $C$ of $G-V(H)$, there exists a bag $B$ in $\mathcal{D_H}$ such that $N_{G}(V(C))$ is contained in the union of the parts of $\mathcal P_H$ contained in $B$.  We will say that the induced subgraph $H$, the partition $\mathcal P_H$ and the path-decomposition $\mathcal D_H$ are \defin{associated} to this decomposition.    
We define the $\mathit{k}$-\defin{partition-pathwidth of} $\mathit{(G, S)}$ as the minimum width of a path-decomposition associated to a $k$-partition-path-decomposition of $(G, S)$, and we denote it by \defin{$\ppw(k, G, S)$}.  Observe that for every $k'\leq k$, we have $\ppw(k', G, S)\geq \ppw(k, G, S)$.  We will use this fact later.    

Now we move on to the proof of \cref{apexforest}.  Doing so, we will prove the following analogue of \cref{tree}, in the setting of $S$-rooted models.  The arguments are strongly inspired by the proof of \cref{badtree} in \cite{dujmovic2024excluded} and our proof of \cref{tree}: 
\begin{theorem}\label{srooted} 
    Let $G$ be a graph, let $S\subseteq V(G)$, and let $T$ be a $t$-vertex tree with $t \geq 2$, maximum degree $d$, and radius $h$.  If $G$ does not contain any $S$-rooted model of $T$, then $\ppw((2t-1)d, G, S)\leq 2h-1$.
\end{theorem}

First, we prove an analogue of \cref{metalemme} (which in fact implies \cref{metalemme} in the case $S=V(G)$).  The proof is an adaptation of the proof of Lemma $8$ in \cite{dujmovic2024excluded}. 

\begin{lemma}
    Let $G$ be a connected graph, let $S\subseteq V(G)$, let $x, y$ be integers with $x\geq 1$ and $y \geq 0$, let $r\in V(G)$, and let $R:= N_G(r)$.  Suppose that there exists a set $X\subseteq V(G) \backslash \{r\}$ of size at most $x$ such that:
    \begin{enumerate}
        \item For every $v\in X$, there exists a path in $G$ from $r$ to $v$ internally disjoint from $X$.
        \label{c1}
        \item For every connected component $C$ of $G-X-r$ containing a vertex of $R$, we have $\ppw(x, C, S\cap V(C))\leq y$.
        \label{c2}
        \item For every minor $G'$ of $G$ with $|V(G')|<|V(G)|$ and such that there exists a model $\mathcal M$ of $G'$ in $G$ and a vertex $r'$ of $G'$ whose branch set in $\mathcal M$ contains $r$, letting $S'$ be the set of vertices of $G'$ whose branch sets in $\mathcal M$ contain at least one vertex of $S$, there exists an $x$-partition-path-decomposition of $(G', S')$ of width at most $y+2$ such that the associated partition contains $\{r'\}$ and the first bag of the associated path-decomposition contains $\{r'\}$.
        \label{c3}
    \end{enumerate}
    Then there exists an $x$-partition-path-decomposition of $(G, S)$ of width at most $y+2$ such that the associated partition contains $\{r\}$ and the first bag of the associated path-decomposition contains $\{r\}$.
    \label{metalemmeproduit}
\end{lemma}

\begin{proof} 
    First, we note that if $R$ is empty then $G$ consists of only the vertex $r$, and it is easy to define an $x$-partition-path-decomposition of $G$ with the desired properties. Thus, we may assume that $R$ is not empty. 

    Next, we deal with the case where $X$ is empty. In this case, every connected component of $G-X-r=G-r$ contains a vertex of $R$ since $G$ is connected and $R=N_G(r)$. Property~\ref{c2} implies then that $\ppw(x, G-r, S\setminus \{r\})\leq y$.  
    Taking an $x$-partition-path-decomposition of $G-r$ of width at most $y$, and adding $\{r\}$ to the associated partition and to all the bags of the associated path-decomposition, we obtain an $x$-partition-path-decomposition of $(G, S)$ of width at most $y+1$ such that the associated partition has $\{r\}$ as one of its parts, and the first bag of the associated path-decomposition contains $\{r\}$. The result follows.  Thus, for the remainder of the proof, we may assume that $X$ is nonempty.
    
    Let $G_1, \ldots , G_p$ be the connected components of $G-X-r$ that contain a vertex of $R$.  By Property~\ref{c2}, for each  $i\in [p]$, there exists $A_i\subseteq V(G_i)$, $\mathcal P_i$ a partition of $A_i$ of width at most $x$, and a path-decomposition $\mathcal{D}_i$ of $G[A_i]/\mathcal{P}_i$ such that: 
\begin{enumerate}
        \item $S \cap V(G_i) \subseteq A_i$; 
        \item $\mathcal D_i$ has width at most $y$; and 
        \item for every connected component $C$ of $G_i-A_i$, there exists a bag of $\mathcal D_i$ that contains all the parts of $\mathcal P_i$ intersecting $N_{G_i}(V(C))$. 
    \end{enumerate}
    Let $Z$ be the union of the vertex sets of all connected components of $G-X-r$ having no vertices in $R$.  Consider a vertex $v\in X$, and let $P_v$ be a path from $r$ to $v$ internally disjoint from $X$, which exists by Property~\ref{c1}.  Because the neighbor of $r$ in $P_v$ belongs to $R$, and thus is in $G_i$ for some $i\in [p]$, and since $v$ is the only vertex in $P_v$ which is also in $X$, $P_v-\{v, r\}$ is fully contained in $G_i$, and so $P_v$ avoids $Z$.  Thus, $G[\bigcup_{v\in X}V(P_v)]$ is connected and avoids $Z$.  Let $G'$ be the graph obtained from $G$ by contracting $\bigcup_{v \in X}V(P_v)$ in one vertex $r'$ and deleting all the other vertices of $G$ not in $Z$, so that $V(G')=\{r'\}\cup Z$.  Observe that $X\cup \{r\}\subset \bigcup_{v\in X}V(P_v)$.  Because $X$ is nonempty, $|X\cup \{r\}|\geq 2$, thus $|V(G')|<|V(G)|$, and hence by Property~\ref{c3}, there exists $A'\subseteq V(G')$, a partition $\mathcal P'$ of width at most $x$ of $A'$ having $\{r'\}$ as one of its parts, and a path-decomposition $\mathcal{D}'$ of $G[A']/\mathcal{P}'$ such that: \begin{enumerate}
        \item $(S\cap Z) \cup \{r'\} \subseteq A'$ ;
        \item $\mathcal D'$ has width at most $y+2$, and its first bag contains $\{r'\}$; and 
        \item for every connected component $C$ of $G'-A'$, there exists a bag of $\mathcal D'$ that contains all the parts of $\mathcal P'$ intersecting $ N_{G'}(V(C))$. 
    \end{enumerate}
Let 
\begin{align*}
    A &:= \{r\} \cup X \cup (\cup_{1\leq i\leq p} A_i)\cup (A'\setminus\{r'\})
\intertext{and}
    \mathcal P &:= \big\{\{r\}, X\big\} \cup (\cup_{1\leq i\leq p} \mathcal P_i)\cup (\mathcal P'\setminus \big\{\{r'\}\big\}).
\end{align*}
Observe that $S\subseteq A$ and that $\mathcal P$ is a partition of $A$ of width at most $x$. 

Let $\mathcal D$ be the sequence of subsets of vertices of $G[A]/\mathcal P$ obtained from the concatenation of $\{\{r\}, X\}, \mathcal{D}_1, \ldots, \mathcal{D}_p$ and $\mathcal D'$ by adding $\{r\}$ and $X$ to every bag coming from the $\mathcal D'_i$'s and replacing $\{r'\}$ by $X$ in every bag of $\mathcal D'$ containing $\{r'\}$.  We argue that $\mathcal D$ is a path-decomposition of $G[A]/\mathcal P$.  It is easy to check that every part of $\mathcal P$ appears in consecutive bags. Indeed, the $\mathcal D_i$'s and $\mathcal D'$ are path-decompositions, $X$ was added to all the bags coming from the $\mathcal D_i$'s, the only part possibly in common between the bags coming from the $\mathcal D_i$'s and the bags coming from $\mathcal D'$ is $X$, and $X$ is also in the first bag coming from $\mathcal D'$. 

Then, we show that every part of $\mathcal P$ is in at least one bag of $\mathcal D$.  This is true by construction for $\{r\}$ and $X$. Every part of $\mathcal P$ which is in $\mathcal P'_i$ for some $i\in [p]$ is in at least one bag of $\mathcal D_i$, and thus in at least one bag of $\mathcal D$.  Every part of $\mathcal P$ which is in $\mathcal P'$ appears in at least one bag of $\mathcal D'$ and, since $\{r'\}\not \in \mathcal P$, appears then also in at least one bag of $\mathcal D$.  Thus, every part of $\mathcal P$ appears indeed in at least one bag of $\mathcal D$.

Next, we show that for every two distinct parts $P,P'\in \mathcal P$ that are adjacent in $G[A]/\mathcal{P}$, there is a bag of $\mathcal D$ containing both $P$ and $P'$. By construction, one of the following holds: 
\begin{enumerate}
    \item $P P' \in E(G[A_i]/\mathcal{P}_i)$ for some $i\in [p]$, and then there is a bag $W$ of $\mathcal D_i$ with $P,P' \in W$, and by construction $W \cup \{\{r\},X\}$ is a bag in $\mathcal D$.
    \item $PP' \in E(G[A']/\mathcal{P}')$, and so there is a bag $W$ of $\mathcal D'$ with $P,P' \in W$, and by construction $W$ or $(W \setminus \{r'\}) \cup \{X\}$ is a bag in $\mathcal{D}$, and contains both $P$ and $P'$.
    \item $P=X$ or $P'=X$, and assume without loss of generality that $P=X$. If $P' \in \mathcal P_i$ for some $i \in [p]$, then any bag $W$ of $\mathcal{D}_i$ containing $P'$ will yields a bag $W \cup \{\{r\},X\}$ in $\mathcal{D}$, which contains both $P$ and $P'$. Similarly, if $P'=\{r\}$, then $\{\{r\}, X\}$ contains both $P$ and $P'$. If $P' \in \mathcal P'$, then $\{r'\} P' \in E(G[A']/\mathcal{P}')$, and so there is a bag $W$ of $\mathcal D'$ with $\{r'\}, P' \in W$, and therefore $\big(W \setminus \big\{\{r'\}\big\}\big) \cup \{X\}$ is a bag in $\mathcal{D}$ containing both $P$ and $P'$.
    \item $P=\{r\}$ or $P'=\{r\}$, and assume without loss of generality that $P = \{r\}$. Then either $P'=\{X\}$, but this case has been already considered, or $P' \in \mathcal{P}_i$ for some $i \in [p]$. Then, any bag $W$ of $\mathcal{D}_i$ containing $P'$ yields a bag $W \cup \{\{r\},X\}$ in $\mathcal{D}$ which contains both $P$ and $P'$.
\end{enumerate}

Finally, we have to show that for each connected component $C$ of $G-A$, there is a bag in $\mathcal D$ that contains every part of $\mathcal P$ intersecting $N_G(V(C))$.  Since $X\cup \{r\}\subseteq A$, $C$ is either a connected component of $G_i-A_i$ for some $i \in [p]$, or $C$ is a connected component of $G'-A'$.  In the first case, there is a bag $W$ of $\mathcal{D}_i$ whose union contains $N_{G_i}(V(C))$, and since $W \cup \{\{r\},X\}$ is a bag of $\mathcal{D}$ and $N_G(V(C)) \subseteq N_{G_i}(V(C)) \cup \{r\}\cup X$, we are done. In the second case, either $r' \in N_{G'}(V(C))$ and so $N_G(V(C)) \subseteq (N_{G'}(V(C)) \setminus \{r'\}) \cup X$, or $N_G(V(C)) = N_{G'}(V(C))$.
In both cases, any bag $W$ of $\mathcal{D}'$ whose union contains $N_{G'}(V(C))$ will yield a bag in $\mathcal{D}$ (namely $\big(W \setminus \big\{\{r'\}\big\}\big)\cup \{X\}$ or $W$) whose union contains $N_{G'}(V(C))$.  

It follows that  $\mathcal D$ is a path-decomposition of $G[A]/\mathcal P$, as claimed. Now, observe that $\mathcal D$ has width at most $y+2$, by construction. Therefore, the set $A$, the partition $\mathcal P$, and the path-decomposition $\mathcal D$ all together give an $x$-partition-path-decomposition of $(G, S)$ of width at most $y+2$ such that the associated partition contains $\{r\}$ and the first bag of the associated path-decomposition contains $\{r\}$, as desired. This concludes the proof.
\end{proof}

We will also use the following lemma, inspired by Lemma~$23$ in \cite{hodor2024quickly}.

\begin{lemma}
    Let $G$ be a graph, let $S, R \subseteq V(G)$, and let $F$ be a nonempty rooted forest.  
    Let $d$ be the number of connected components of $F$ and let $p:=\pw(G, S)$.   Assume that $G$ does not contain any $(S, R)$-rooted model of $F$. 
    Then, there exists a set $X\subseteq V(G)$ and a connected component $T$ of $F$ such that $|X|\leq (d-1)(p+1)$, and $G-X$ does not contain any  $(S, R)$-rooted 
    model of $T$. 
    \label{erdos_posa_forest_general}
\end{lemma}
\begin{proof}
    The proof is by induction on $d$.  If $d=1$, then the result is true by taking $X=\emptyset$ and $T=F$.

    Next, we do the inductive step. Assume thus $d\geq 2$, and that the lemma holds for smaller values of $d$. 
    If for some connected component $T$ of $F$, the graph $G$ has no $(S, R)$-rooted model of $T$, then the lemma holds with $X=\emptyset$ and $T$. Thus, we may assume that, for every connected component $T$ of $F$, there is some $(S, R)$-rooted model of $T$ in $G$. 
    (Note that this implies in particular that $S$ and $R$ are both non empty.) 
    We may also assume that every connected component of $G$ contains at least one vertex from $S$, since those that do not can be freely discarded as they have no impact on the existence of the set $X$.   
    
    Consider a path-decomposition of $(G, S)$ of width at most $p$, with $H$ denoting the corresponding induced subgraph of $G$, and $B_1, \dots, B_m$ the corresponding sequence of bags. Observe that for every connected component $C$ of $G-V(H)$ we have $N_{G}(V(C))\ne \emptyset$ (since every connected component of $G$ contains at least one vertex from $S$) and moreover $N_{G}(V(C)) \subseteq B_i$ for some $i\in [m]$ (by the definition of a path-decomposition of $(G, S)$). For every $i\in [m]$, let $V_i$ be the union of the vertex sets of the connected components $C$ of $G-V(H)$ such that $N_{G}(V(C)) \subseteq B_i$. Thus, $V(H)\cup (\bigcup_{i=1}^m V_i)=V(G)$.  
    For every $j\in [m]$, let $G_j := G[\bigcup_{i=1}^j (B_i \cup V_i)]$. 
    Observe that, by the definition of a path decomposition of $(G, S)$, the set $B_j$ separates $V(G_j)$ from $V(G)\setminus V(G_j)$ in $G$.
    Let $j\in [m]$ be minimum such that $G_j$ contains an 
    $(S, R)$-rooted
    model of some connected component $T'$ of $F$.  
    Note that $j$ is well defined by our assumption on $G$, since $G=G_m$ contains an 
    $(S, R)$-rooted 
    of some connected component of $F$.      
    Let $F':=F-V(T')$.  
    
    Observe that $G-V(G_j)$ does not contain an $(S, R)$-rooted model of $F'$.  Indeed, if it would be the case, then $G$ would contain an $(S, R)$-rooted model of $F$,  because $G_j$ contains an 
    $(S, R)$-rooted model of $T'$, a contradiction.
    This implies that $G-B_j$ does not contain any $(S, R)$-rooted model of $F'$, because $G_j-B_j$ does not contain any $(S, R)$-rooted model of any tree of $F'$. 
    
    Thus, we may apply the induction hypothesis on $G-B_j$ and $F'$.  Letting $p':=\pw(G-B_j, S\setminus B_j)$, this gives a connected component $T$ of $F'$ and a set $X'\subseteq V(G-B_j)$ of size at most $(d-2)(p'+1)\leq (d-2)(p+1)$ such that $G-B_j-X'$ does not contain any 
    $(S, R)$-rooted 
    model of $T$.  Then, letting $X:= X'\cup B_j$, it follows that $|X|=|X'|+|B_j|\leq (d-2)(p+1)+(p+1)=(d-1)(p+1)$, and that $G-X$ does not contain any 
    $(S, R)$-rooted 
    model of $T$, as desired. This concludes the proof.
\end{proof}

Next, we show an analogue of \cref{lemmearticle} (which implies \cref{lemmearticle} in the case $S=V(G)$ but with a worse bound for the width of the partition).  

\begin{lemma} Let $T$ be a $t$-vertex tree rooted tree of height $h$ and maximum degree $d$, where $t\geq 2$ and $h, d \geq 1$.   
Suppose that  $G$ is a connected graph with a distinguished set $S\subseteq V(G)$ and vertex $r\in V(G)$ such that $G$ does not contain any $(S, \{r\})$-rooted model of $T$. Then, there exists a $((2t-1)d)$-partition-path-decomposition of $(G, S)$ of width at most $2h-1$ such that the associated partition contains $\{r\}$ and the first bag of the path-decomposition contains $\{r\}$.
    \label{lemmetechniquepartitions}
\end{lemma}
\begin{proof}
    The proof is by induction on the pairs $(h, |V(G)|)$, in lexicographic order.  If $|V(G)|=1$, the lemma is easily seen to hold by taking the partition $\{\{r\}\}$ of $V(G)$. Suppose thus $|V(G)|\ge 2$. 
      
    First, we consider the case $h = 1$.
    Let $A:=S \cup \{r\}$ and let $G'$ be the graph obtained from $G$ by contracting  each induced subgraph of $G$ corresponding to a connected component of $G-A$ into one vertex. 
    Observe that $G'$ is connected, since $G$ is connected.  
    For every nonnegative integer $i$, let $V'_i:= \{ v\in V(G') \ | \ \text{dist}_{G'}(v, r) =i\}$ 
    and $V_i:=V'_i\cap A$. 
    Observe that if $|V_i|\geq d+1$ for some $i\geq 1$, then contracting the connected subgraph of $G'$ induced by $V'_0 \cup V'_1 \cup \cdots \cup V'_{i-1}$ into a single vertex, we see that there is an $(S, \{r\})$-rooted model of $T$ in $G'$,
    and thus also in $G$, a contradiction.  Thus, $|V_i|\leq d$ holds for every $i\geq 0$.  
    
    Let $m\in \mathbb N$ be smallest such that $V_{2(m+1)}= \emptyset$.  Let $\mathcal{P} := \{\{V_{2i} \cup V_{2i+1} \} \ | \ i\in  \mathbb N, i \leq m\}$.  If $m\ne 0$, let $D_i:= \{\{V_{2i} \cup V_{2i+1}\}, \{V_{2i+2} \cup V_{2i+3}\}\}$ for every nonnegative integer $i < m$,
    and let $\mathcal{D} := \{D_i \ | \ i\in  \{0, 1, \dots, m-1\}\}$.  If $m=0$, let $D_0 := \{\{V_0, V_1\}\}$ and $\mathcal{D}:=\{D_0\}$.
    Because $G[A]/\mathcal P$ is a subgraph of a path, it is easy to see that $\mathcal D$ is a path-decomposition of $G[A]/ \mathcal P$ of width at most $1$.  If $m\ne 0$, for each connected component $C$ of $G-A$, $C$ corresponds to a vertex $v'$ of $G'$, and there exists exactly one index $i$ between $1$ and $m-1$ such that $v'\in V'_i$.  Thus, each vertex of $A$ that has a neighbor in $C$ in $G$ belongs to either $V_{i-1}, V_{i}$ or $V_{i+1}$, which are all contained in a common bag of $\mathcal D$.  Moreover, each part of $\mathcal P$ has size at most $2d\leq (2t-1)d$.  The case $m=0$ is easier, and very similar.  Thus, the lemma holds for $h=1$. 
    
    Now we do the inductive step.  Assume that $h\geq 2$. Let $R:= N_G(r)$.  Let $F$ be the rooted forest obtained from $T$ by removing its root and adding a single-vertex component.  Observe that $F$ has at most $d+1$ connected components, and that if $G-r$ contains an $(S, R)$-rooted model of $F$, then $G$ contains an $(S, \{r\})$-rooted model of $T$, a contradiction.    
    Thus, $G-r$ has no $(S, R)$-rooted model of $F$.  
    
    Note also that $G-r$ has no $(S, R)$-rooted model of $T$, since such a model would easily give an $(S, \{r\})$-rooted model of $T$ in $G$. 
    \cref{diestelSrooted} implies then that $\pw(G-r, S\setminus\{r\})\leq 2t-2$.
    
    It follows from the above observations that we may apply \cref{erdos_posa_forest_general} with $G-r, S\setminus\{r\}, $R$, p=2t-2$ and $F$, and thus there exists a set $X\subseteq V(G-r)$ of size at most $(2t-1)d$ and a tree $T'$ of $F$ such that $G-r-X$ does not contain any 
    $(S, R)$-rooted $T'$-model.  Observe that, because $h\geq 2$, there is a tree of $F$ whose height is at least $1$, and thus we may assume that $T'$ has been chosen so that its height is at least $1$.   Without loss of generality, assume that $X$ is inclusion-wise minimal such that $G-r-X$ does not contain any $(S, R)$-rooted $T'$-model.  Let $d'$ be the maximum degree of $T'$. Observe that $d'\leq d$.  We will show that the hypotheses of \cref{metalemmeproduit} are satisfied with $S, x=(2t-1)d$, $y=2h-3$, $r$, and $X$.

    We start by showing that condition \ref{c1} of \cref{metalemmeproduit} holds. 
    Let $v\in X$.  By the minimality of $X$, there exists an $(S, R)$-rooted model of $T'$ in $(G-r)-(X\backslash \{v\})$.  Because $G-r-X$ does not contain any $(S, R)$-rooted model of $T'$, this model has to contain $v$ in at least one of its branch sets, and thus in particular there is a path in $(G-r)-(X\setminus \{v\})$ between $v$ and some vertex $r'\in R$.  Thus, in $G-(X \setminus \{v\})$, there is a path from $v$ to $r$, as desired. 

    Next, we show that condition \ref{c2} of \cref{metalemmeproduit} holds. 
    Let $C$ be a connected component of $G-X-r$ containing a vertex $r'\in R$.  By hypothesis on $X$, $C$ does not contain any $(S, R)$-rooted $T'$-model.  Because the height $h'$ of $T'$ satisfies $1\leq h'\leq h-1$, we may apply the induction hypothesis to $T', h', d', C, S\cap V(C)$ and $r'$, and it follows in particular that 
    \[
    \ppw((2t-1)d, C, S\cap V(C)) 
    \leq \ppw((2|V(T')|-1)d', C, S\cap V(C)))
    \leq 2(h-1)-1
    = 2h-3
    \]        
    using that $d'\leq d$ and $|V(T')|\leq t$.  Thus, condition \ref{c2} is verified. 

    Finally, we show that condition \ref{c3} of \cref{metalemmeproduit} holds. 
    Let $G'$ be a minor of $G$ such that $|V(G')|<|V(G)|$, and such that there exists a model $\mathcal M$ of $G'$ in $G$ and a vertex $r'$ in $G'$ whose branch set in $\mathcal M$ contains $r$.  Let $S'$ be the set of vertices of $G'$ whose branch set in $\mathcal M$ contains a vertex of $S$.  Because $|V(G')|<|V(G)|$, and because $G'$ does not contain any $(S', \{r'\})$-rooted model of $T$---otherwise $G$ would contain an $(S, \{r\})$-rooted model of $T$---we may apply the induction hypothesis.  
    Therefore, there exists a $((2t-1)d)$-partition-path-decomposition of $(G', S')$ of width at most $2h-1$ such that the associated partition contains $\{r'\}$ and the first bag of the path-decomposition contains $\{r'\}$.  
    Thus, condition \ref{c3} is also verified. 
    
    Therefore, we may apply \cref{metalemmeproduit} as claimed, and the lemma follows.
\end{proof}

The proof of \cref{srooted} follows directly from applying \cref{lemmetechniquepartitions} on each connected component of the graph, by rooting $T$ so that its height is equal to its radius, and picking an arbitrary vertex $r$ in each connected component. 

Next, we will show \cref{apexforest}. Its proof relies on \cref{srooted} and a couple extra lemmas. 
We start with the following lemma. 
We remark that its proof is an adaptation of a proof idea that appears inside the proof of Lemma~15 in \cite{hodor2024quickly}.

\begin{lemma}
    Let $G$ be a connected graph, let $u\in V(G)$, let $S:=N_G(u)$ and let $x, y\in \mathbb N$.  
    Suppose that there exists an $x$-partition-path-decomposition $\mathcal D$ of width at most $y$ of $(G-u, S)$, with associated partition $\mathcal P$ and associated induced subgraph $H$. 
    Suppose further that $H$ has been chosen so that $|V(H)|$ is minimal.   
    Then, for each connected component $C$ of $G-u - V(H)$, the graph $G-V(C)$ is connected. 
 \label{connexite}
\end{lemma}

\begin{proof}
    Suppose for contradiction that $G-V(C)$ is not connected for some connected component $C$ of $G-u - V(H)$.  
    This implies that there exists a connected component $C'$ of $G - V(C)$ such that $V(C')$ is disjoint from $\{u\} \cup N_{G}(u)$.  
    Since $G$ is connected, there exists an edge $vw \in E(G)$ such that $v\in V(C)$ and $w\in V(C')$.  
    Because $C$ is a connected component of $G-u - V(H)$ and $w \in N_G(V(C))$, 
    we have $w\in \{u\}\cup V(H)$. 
    Since $w \in V(C')$ and $u \not\in V(C')$, we deduce that $w \in V(H)$.

    Let $A:= V(H)\setminus V(C')$.
    Because $w\in V(H)\cap V(C')$, we have $|A| < |V(H)|$.
    Let $\mathcal P' := \{P\cap A \mid P\in \mathcal P\}\setminus \{\emptyset\}$.
    Let $\mathcal D'$ be the path-decomposition of $G[A] / \mathcal P'$ obtained from $\mathcal D$ by replacing each
    bag $B$ by $\{P \cap A \mid P \in B\} \setminus \{\emptyset\}$.
    Our goal is to show that $\mathcal D'$ is an $x$-partition-path-decomposition of $(G-u, S)$ of width at most $y$ with associated subgraph $G[A]$ and associated partition $\mathcal P'$. 
    Since $|A| < |V(H)|$, this will contradict the choice of $H$ and conclude the proof of the lemma.  
    The only nontrivial thing to show is that for every connected component $C''$ of $G-u - A$, there exists a bag $B'$ in $\mathcal D'$ such that the union of the sets in $B'$ contains $N_{G-u}(V(C''))$.
     
    Let $C''$ be a connected component of $G-u - A$.  
    If $V(C'')$ is disjoint from $V(H)\setminus A = V(C')\cap V(H)$, then $C''$ is also a connected component of $G-u - V(H)$.  
    Thus, there exists a bag $B\in \mathcal D$ such that $N_{G-u}(V(C'')) \subseteq \bigcup_{P\in B} P$.  
    Because $C''$ is a connected component of $G-u - A$, we have $N_{G-u}(V(C''))\subseteq A$, and thus $N_{G-u}(V(C'')) \subseteq \bigcup_{P\in B} (P\cap A)$.  
    Moreover, $\{P\cap A \mid P\in B\} \setminus \{\emptyset\}$ is a bag of $\mathcal D'$, which concludes this case.

    We consider now the remaining case: $V(C'')$ contains at least one vertex from $V(C') \cap V(H)$. 
    Recall that $C$ and $C'$ are connected subgraphs of $G-u - A$,
    and that there is an edge between $V(C)$ and $V(C')$ in $G$.
    Hence, $V(C) \cup V(C')$ induces a connected subgraph of $G-u - A$.
    Moreover, $N_{G-u}(V(C')) \subseteq V(C)$ (because $C'$ is a connected component of $G-V(C)$), 
    and $N_{G-u}(V(C)) \subseteq V(H) \subseteq V(C') \cup A$ (because $C$ is a connected component of $G-u - V(H)$).
    Therefore, $N_{G-u}(V(C) \cup V(C')) \subseteq A$.
    This implies that $V(C) \cup V(C')$ induces a connected component of $G-u - A$.
    Because $V(C'')$ intersects $V(C')$,
    we deduce that $V(C) \cup V(C') = V(C'')$.
    In particular,
    \[
        N_{G-u}(V(C'')) \subseteq A \cap (N_{G-u}(V(C)) \cup N_{G-u}(V(C'))) \subseteq N_{G-u}(V(C)) \cap A.
    \]
    Since $\mathcal{D}$ is an $x$-partition-path-decomposition of $(G-u,S)$,
    and because $C$ is a connected component of $G-u - V(H)$,
    there is a bag $B$ in $\mathcal D$ such that $N_{G-u}(V(C))\subseteq \bigcup_{P\in B} P$.
    Therefore, we conclude that $N_{G-u}(V(C'')) \subseteq N_{G-u}(V(C)) \cap A \subseteq \bigcup_{P \in B} (P \cap A)$.
    Since $\{P \cap A \mid P \in B\} \setminus \{\emptyset\}$ is a bag in $\mathcal{D}'$, 
    this concludes the proof that $\mathcal{D}'$ is an $x$-partition-path-decomposition of $(G-u,S)$ of width at most $y$. 
\end{proof}

Now we prove the following technical lemma, which directly implies \cref{apexforest}.

\begin{lemma}\label{technical_for_apexforest}
    For every tree $T$ with $t\geq 2$ vertices, radius $h$, and maximum degree $d$, 
    for every graph $G$ not containing $T^+$ as a minor, 
    and for every vertex $u\in V(G)$, there exists a partition $\mathcal P$ of $V(G)$ of width at most $2(t-1)d$ 
    such that 
    $ \{ u \} \in \mathcal P$, and 
    $G/\mathcal P$ has a tree-decomposition of width at most $4h-1$ such that each bag containing $\{u\}$ has size at most $2h+1$.
\end{lemma}

\begin{proof}
    The proof is by induction on $|V(G)|$. The base case $|V(G)|=1$ is trivial.
    Now suppose $|V(G)|>1$.

    Now we do the inductive case.  If $G$ is not connected, we may apply the induction hypothesis on each connected component of $G$.  So we may assume that $G$ is connected.  
    Let $S:= N_G(u)$. 
    If $G-u$ contains an $S$-rooted model of $T$, then this model together with $\{u\}$ yields a model of $T^+$ in $G$, a contradiction.  
    So $G-u$ has no $S$-rooted model of $T$, and so by \cref{srooted}, 
    there exists a $(2t-1)d$-partition-path-decomposition of $(G-u,S)$ of width at most $2h-1$.  
    Let $H$, $\mathcal P_H$ and $\mathcal D_H$  be, respectively, the induced subgraph of $G-u$, the partition of $V(G-u)$ and the path-decomposition associated to this decomposition.   
    For the purpose of this proof, it will be convenient to see the path-decomposition not as a sequence of bags (as in previous proofs) but as a tree-decomposition where the tree indexing the decomposition is a path; we let $P$ denote this path. 
    We choose such a tuple $(H, \mathcal{P}_H, \mathcal{D}_H, P)$ with $|V(H)|$ minimum.
    Let $C_1, C_2, \ldots, C_m$ be the connected components of $G-u - V(H)$.  
    By assumption, for each $i \in \{1, \dots, m\}$, there exists a bag $B_i$ in $\mathcal D_H$ such that the union of the sets in $B_i$ contains $N_{G-u}(V(C_i))$, and even $N_G(V(C_i))$ since $N_G(u)\subseteq V(H)$.  
    We remark that $B_i$ and $B_j$ may possibly refer to the same bag for $i\ne j$.  
    
    Let $i \in \{1, \dots, m\}$.
    By the minimality of $|V(H)|$ and by \cref{connexite}, $G-V(C_i)$ is connected.  
    Let $G'_i$ be the graph obtained from $G$ by contracting the connected subgraph induced by $G-V(C_i)$ into one vertex $u_i$.  
    Thus, $V(G'_i)=V(C_i)\cup \{u_i\}$ and $G'_i$ is a minor of $G$, and in particular $G'_i$ has no $T^+$ minor. 
    Since $G$ is connected and $|V(G)|\geq 2$, it follows that $S$ and $V(H)$ are not empty, and hence  $|V(G'_i)|< |V(G_i)|$.
    Hence, by the induction hypothesis, there exists a partition $\mathcal{P}_i$ of $V(G'_i)$ of width at most $(2t-1)d$ and such that $\{u_i\} \in \mathcal P_i$, and $G'_i/\mathcal{P}_i$ has a tree-decomposition $\mathcal D_i$ of width at most $4h-1$ with associated tree $Y_i$ such that each bag containing $\{u_i\}$ has size at most $2h+1$. 
    Let $\mathcal D'_i$ be the tree-decomposition of the graph induced in $G$ by $V(C_i) \cup \bigcup_{Q \in B_i} Q$ 
    obtained from $\mathcal D_i$ by replacing every bag $B$ containing $\{u_i\}$ by $\big(B \setminus \big\{\{u_i\}\big\}\big) \cup B_i$.  
    (We keep the same tree associated to the tree-decomposition.) 
    For every bag $W'$ of $\mathcal{D}'_i$, either $W'$ is a bag in $\mathcal{D}_i$ and so has size at most $4h$,
    or $W' = \big(W \setminus \big\{\{u_i\big\}\}\big) \cup B_i$ for some bag $W$ of $\mathcal{D}_i$ with $\{u_i\} \in W$.
    But then, $|W| \leq 2h+1$, and so $|W'| \leq (2h+1)-1 + |B_i| \leq 4h$.
    Hence, $\mathcal D'_i$ has width at most $4h-1$.

    Now define $\mathcal P := \mathcal P_H \cup (\bigcup_{i=1}^m \mathcal P_i) \setminus \big(\bigcup_{i=1}^m \big\{\{u_i\}\big\}\big)$, 
    and let $\mathcal D$ and $Y$ be the tree-decomposition of $G/\mathcal P$ and the tree associated to $\mathcal D$, obtained from $\mathcal D_H$, $P$, the $\mathcal D'_i$'s and the $Y_i$'s in the following way:
    
    \begin{enumerate}
        \item Add $\{u\}$ to all the bags of $\mathcal D_H$.  
        \item For each $i\in \{1, \dots, m\}$, link one vertex of $Y_i$ corresponding to a bag of $\mathcal D_i$ containing $\{u_i\}$ to a vertex of $P$ whose corresponding bag is $B_i$.  This gives the tree $Y$. 
    \end{enumerate}
    
     Let us now show that $\mathcal P$ and $\mathcal D$ have the properties stated in the lemma.  By construction, $\mathcal P$ is a partition of $V(G)$, and $\{u\}\in \mathcal P$. 
     
     Because every part of $\mathcal P$ is either $\{u\}$, or a part of $\mathcal P_H$, or a part of $\mathcal P_i$ for some $i\in \{1, \dots, m\}$, we have by induction that each part of $\mathcal P$ has size at most $(2t-1)d$.
     
    All the bags $W$ of $\mathcal{D}$ containing $\{u\}$ are of the form $W = W_H \cup \big\{\{u\}\}$ for some bag $W_H$ of $\mathcal{D}_H$, and so $|W| \leq 2h+1$.
            Every other bag of $\mathcal{D}$ is a bag of $\mathcal{D}'_i$ for some $i \in \{1, \dots, m\}$, and so has size at most $4h$ since $\mathcal{D}'_i$ has width at most $4h-1$.            
            
    The only parts of $\mathcal P$ that potentially appear in bags coming from multiples decompositions among $\mathcal D_H$ and the $\mathcal D'_i$'s are bags that appear in $\mathcal D_H$, so it is easy to check that, by construction, for every part of $\mathcal P$, the vertices of $Y$ corresponding to bags of $\mathcal D$ containing this part span a subtree of $Y$. Hence, $\mathcal{D}$ is a tree-decomposition of $G/\mathcal P$, as claimed.  
    
    This proves that $\mathcal{P}$ and $\mathcal{D}$ have the desired properties, and concludes the proof of the lemma.
\end{proof}

\section{Open questions} \label{sec:open}
Some of the bounds we have provided are not known to be tight:
\begin{enumerate}
    \item The bound of $t-2$ on the size of the clique in the blow-up in \cref{tree} is close to optimal but there might still be a small room for improvement; we are only aware of a $\Omega(\frac t h)$ lower bound, as explained in the introduction. 
    \item The bound of $(2t-1)d$ on the size of the clique in the blow-up in \cref{apexforest} is likely not optimal. We expect that $O(t)$ bound should hold in this setting as well.  
    Again, we only know of a lower bound of $\Omega(\frac t h)$. 
    \item We do not know if the bound of $4h-1$ on the treewidth in \cref{apexforest} is optimal, but we know that it should be at least $2h$, as mentioned in the introduction. 
    We expect the lower bound to be closer to the truth. 
\end{enumerate}

\bibliographystyle{plainnat}
\bibliography{Bibliography}

\appendix
\section{Short proof of Theorem~\ref{leafseymour-improved}}

Let us emphasize again that this short proof is heavily based on  the proofs in \cite{leaf2015tree}, in particular the proofs of statements 4.3 and 4.4 in that paper: It is essentially a copy of these proofs with some parts removed, and with some easy adaptations. To emphasize the similarities, we chose to stay as close as possible to the presentation in \cite{leaf2015tree}. 

First, we need to recall some definitions. 
For $k\in \mathbb N$, a \defin{bramble} of order $k$ in a graph $G$ is a set $\mathcal B$ of non-empty connected subgraphs of $G$, such that 
\begin{enumerate}
    \item every two members $B_1, B_2 \in \mathcal B$ touch, that is, either $V(B_1) \cap V(B_2)\ne \emptyset$, or there is an edge of $G$ with one end in $V(B_1)$ and the other in $V(B_2)$,
    \item for every $X \subseteq V (G)$ with $|X| < k$, there exists $B \in \mathcal B$ with $X \cap V(B) = \emptyset$.
\end{enumerate}
It was proved in \cite{seymour1993graph} that for every positive integer $k$, a graph has treewidth at least $k-1$ if and only if it has a bramble of order $k$.

A \defin{separation} of $G$ is a pair $(A, B)$ of subsets of $V(G)$ such that $A\cup B = V(G)$ and every edge of $G$ is either contained in $A$ or contained in $B$.
The \defin{order} of $(A,B)$ is |$A \cap B|$.
For $X,Y \subset V(G)$, an \defin{$X$--$Y$ path} is a path in $G$ that is either a one-vertex path with the vertex in $X \cap Y$ or a path with one endpoint in $X$ and the other endpoint in $Y$ such that no internal vertices are in $X \cup Y$. 

If $(A, B)$ is a separation of $G$, we say that $(A, B)$ \defin{left-contains} a model $\{W_v \subseteq V(G) \mid v \in V(H)\}$ of $H$ if $|A \cap B| = |V(H)|$, for every $v\in V(H)$, $W_v\subseteq A$, and $|W_v \cap A\cap B|=1$. If such a model exists we say that $(A, B)$ \defin{left-contains} $H$.

We also need Menger's Theorem:  

\begin{theorem}[Menger's Theorem]
    Let $G$ be a graph and $X,Y \subset V(G)$.
    There exists a separation $(A,B)$ of $G$ such that $X \subset A$, $Y \subset B$, and there exists $|A \cap B|$ pairwise disjoint $X$--$Y$ paths.
\end{theorem}

We will show the following lemma, which directly implies \cref{leafseymour-improved} (by letting $G[B]-A$ model the apex vertex): 
\begin{lemma}
Let $w \geq 1$ be an integer, let $T$ be a tree with $|V (T)| = w$, and let $G$ be a graph with
treewidth at least $w$. Then there is a separation $(A, B)$ of G such that
\begin{enumerate}
    \item $|A\cap B|=w$,
    \item $G[B]-A$ is connected and every vertex in $B\cap A$ has a neighbor in $B\setminus A$, and
    \item $(A, B)$ left-contains $T$.
\end{enumerate}
\end{lemma}
\begin{proof}
Choose a vertex $t_1$ of $T$, and number the other vertices $t_2, \ldots ,t_w$ in such a
way that for $2 \leq i \leq w$, $t_i$ is adjacent to one of $t_1, \ldots ,t_{i-1}$. For $1 \leq i \leq w$, let $T_i$ be the subtree of $T$ induced by ${t_1, \ldots ,t_i}$. Now let $G$ be a graph with treewidth at least
$w$; we know that it has a bramble $\mathcal B$ of order at least $w+1$. Choose $\mathcal B$ maximal; thus if $C$ is a connected subgraph of $G$ including a member of $\mathcal B$, then $C \in \mathcal B$ (because otherwise it could be added to $\mathcal B$, contrary to maximality). For each $X \subseteq V (G)$ with $|X| \leq w$,
there is therefore a unique component of $G - X$ that belongs to $\mathcal B$; let its vertex set be $\beta(X)$. 

Choose $v \in \beta(\emptyset)$; then $\beta(\{v\})\subseteq \beta(\emptyset)$, and the separation $(V(G) \setminus (\beta(\emptyset)\setminus\{v\}), \beta(\emptyset))$ left-contains $T_1$. Consequently we may choose a separation $(A, B)$ of $G$ with the following properties:
\begin{enumerate}
    \item $(A, B)$ has order at least one, and at most $w$; say order $k$ where $1 \leq k \leq w$,
    \item $(A, B)$ left-contains $T_k$,
    \item $\beta(A\cap B)\subseteq B$,
    \item there is no separation $(A', B')$ of $G$ of order strictly less than $k$, with $A \subseteq A'$ and $B' \subset B$, and such that $ \beta(A'\cap B')\subseteq B'$
    \item subject to these conditions, $|A|-|B|$ is maximum.
\end{enumerate}

\begin{claim}
    There is no separation $(A', B')$ of $G$ of order $k$ with $A\subseteq A'$, $B' \subseteq B$ and $(A, B)\ne (A', B')$ such that $\beta(A'\cap B')\subseteq B'$.
    \label{clm1}
\end{claim}
\begin{proofclaim}
Suppose that there is such a separation $(A', B')$. From the optimality of $(A, B)$, $(A', B')$ does not left-contain $T_k$. 
Consequently there do not exist $k$ vertex-disjoint  $(A\cap B)$--$(A'\cap B')$ paths; and so by Menger's Theorem there is a separation $(C, D)$ of order less than $k$, with $A \subseteq C$ and $B' \subseteq D$. Since $\beta(C\cap D)$ touches $\beta (A'\cap B')$, and $\beta(A'\cap B')\subseteq B'\subseteq D$, it follows that $\beta(C\cap D) \subseteq D$. But this contradicts the fourth condition above. 
\end{proofclaim}

\begin{claim}
 $G[B]-A$ is connected, and every vertex of $B\cap A$ has a neighbor in $B\setminus A$. 
 \label{clm2}
\end{claim}
\begin{proofclaim}
Now $\beta(A \cap B) \subseteq B \setminus A$ and hence is the vertex set of a connected component of $G[B]-A$. Let $D := (A\cap B)\cup \beta(A\cap B) $, and let $C := V(G) \setminus \beta(A\cap  B)$; then $(C, D) $ is a separation of $G$ satisfying the first four conditions above. From the optimality of $(A, B)$ it follows that $(A, B) = (C, D)$, and in particular, $\beta(A \cap B) = B\setminus A$. This proves the first assertion. 

For the second assertion, suppose some vertex $v \in  A \cap B$ has no neighbor in $B\setminus A$. Then $(A, B - \{v\})$ is a separation, and since $\beta(A \cap (B \setminus \{v\}))$ touches $\beta(A \cap B)$, and hence is contained in
$B \setminus \{v\}$, this contradicts the fourth condition above.
\end{proofclaim}
\begin{claim}
    $k=w$
    \label{clm3}
\end{claim}
\begin{proofclaim}
Arguing by contradiction, suppose that $k<w$. Let $\{W_j \mid 1\leq i \leq k\}$ be a model of $T_k$ in $A$, where for every $j\in \{1,\dots, k\}$, $W_j$ is the branch set of $t_j$, and contains a unique vertex $v_j$ of $A \cap B$. Let $i\in \{1,\dots, k\}$ be such that $t_{k+1}$ is adjacent in $T$ to $t_i$. By our second claim, $v_i$ has a neighbor in $B\setminus A$, say $v_{k+1}$. Let $A' := A\cup \{v_{k+1}\}$; then $(A', B)$ is a separation of $G$, and it left-contains $T_{k+1}$ (with the model $\{W_j \mid 1\leq i \leq k\}\cup \{\{v_{k+1}\}\}$). Moreover, since $\beta(A' \cap B)$ touches $\beta(A \cap B)$, it is a subset of $B$; and $(A', B)$ satisfies the fourth condition above because of our first claim. But this contradicts the optimality of $(A, B)$, and hence proves this claim.
\end{proofclaim}

This concludes the proof of the lemma. 
\end{proof}

\end{document}